\documentclass{amsart}

\usepackage[english]{babel}
\usepackage[utf8]{inputenc}

\usepackage{amsmath, amscd, %MRLmakeidx,
  xspace}     %, showlabels}

\usepackage{amssymb,hyperref}
\usepackage{color,fullpage}
\usepackage{enumerate}
\usepackage[all,cmtip]{xy}
\theoremstyle{plain}
\newtheorem{note}[subsection]{Note}

\newtheorem{theo}[subsection]{Theorem}
\newtheorem{prop}[subsection]{Proposition}
\newtheorem{lemma}[subsection]{Lemma}
\newtheorem{cor}[subsection]{Corollary}
\theoremstyle{definition}
\newtheorem{Def}[subsection]{Definition}
\newtheorem{remark}[subsection]{Remark}

\newcommand{\eg}{e.g.,\xspace}

\newcommand{\bQ}{{\mathbb Q}}

\newcommand{\bF}{\mathbb F}

\newcommand{\sab}{{\ensuremath{\mathcal A}}} %the category of simplicial
\newcommand{\Tor}{\operatorname{Tor}}

\newcommand{\smsh}{{\wedge}}

\newcommand{\we}{\overset{\sim}{\rightarrow}}
\newcommand{\ew}{\overset{\sim}{\leftarrow}}
\def\calM{\mathcal{M}}
\def\calN{\mathcal{N}}
\def\calC{\mathcal{C}}

\newcommand{\id}{\mathrm{id}}

\newcommand{\THH}{{T\hspace{-.5mm}H\hspace{-.5mm}H}}
\newcommand{\thh}[1]{\THH^{[#1]}}
\newcommand{\HH}{H\hspace{-.5mm}H}
\newcommand{\hh}[1]{\HH^{[#1]}}
\def\F{\mathbb{F}}
\def\Z{\mathbb{Z}}

\def\ess{\mathbb{S}}
\def\Sp{\text{$\sf{Sp}$}}
\def\ra{\rightarrow}

\begin{document}

\title[On higher topological Hochschild homology]{On higher topological Hochschild homology of rings of integers}
\author{Bj{\o}rn Ian Dundas}
  \address{Department of Mathematics, University of Bergen, 
    Postboks 7803, 5020 Bergen, Norway} 
  \email{dundas@math.uib.no}

\author{Ayelet Lindenstrauss}
  \address{Department of Mathematics, Indiana University, Bloomington IN 47405, USA}
  \email{alindens@indiana.edu}

\author{Birgit Richter}
  \address{Fachbereich Mathematik, Universit\"at Hamburg, Bundesstra{\ss}e 55, 20146 Hamburg, Germany} 
  \email{birgit.richter@uni-hamburg.de}

\begin{abstract}
We determine higher topological Hochschild homology of rings of
integers in number fields with coefficients in suitable residue fields.
We use the iterative description of higher $\THH$ for this and Postnikov
arguments that allow us reduce the necessary computations to calculations
in homological algebra, starting from the results of B\"okstedt
and Lindenstrauss-Madsen on (ordinary) topological Hochschild homology.
\end{abstract}

\maketitle

\section{Introduction}

Factorization homology provides a general framework for associating an
invariant, $\int_M R$, to
a topological $n$-dimensional manifold $M$ and an algebra $R$ over the
Boardman-Vogt little $n$-cubes operad  \cite{BV}. 
For a comprehensive overview of factorization homology see
\cite{af}. Commutative 
algebras are algebras 
over the little $n$-cubes operad for all $n$. In \cite[Proposition 
5.1]{af} it is shown  
that for commutative algebras $R$ the factorization homology, $\int_M R$, 
reduces to $R \otimes U(M)$ where $U(M)$ is  
the underlying topological space of $M$ and $(-) \otimes U(M)$ refers to the
topological enrichment of the corresponding category. 

In the setting of spectra, this shows 
that for a commutative ring spectrum $R$ the factorization homology
with respect to an $n$-sphere, $\mathbb{S}^n$, can be expressed in
classical terms as $R \otimes \mathbb{S}^n$ and this is what we call 
topological Hochschild homology of order $n$ of $R$. 
 
We need a version with coefficients and we use simplicial sets
as combinatorial models for topological spaces. It is shown in
\cite[VII.3.2]{EKMM} that both enrichments (simplicial sets or
topological spaces) give equivalent outcome. 
Topological Hochschild homology of a ring spectrum $R$ with
coefficients in a  module $N$ is obtained as a simplicial 
object where one uses the standard 
simplicial model of the $1$-sphere, $S^1$, and glues
$N$ to the basepoint of $S^1$ and $R$ to all other simplices in
$S^1$. 
Topological Hochschild homology of order $n$ of a commutative ring 
spectrum $R$ with 
coefficients in $N$, $\thh{n}(R,N)$, is the analogue of
this where we use $S^n = (S^1)^{\wedge n}$ as a simplicial model of
the $n$-sphere and glue again $N$ to the basepoint and $R$ to all
other simplices of $S^n$.   For a definition using the topological
circle  $\mathbb{S}^1$ see
\cite{MSV}; for an approach preserving the epicyclic structure, see
\eg \cite{BCD} or \cite{Ve}. 

There are natural stabilization maps
\[ \pi_*(\thh{n}(R,N)) \ra \pi_{*+1}(\thh{n+1}(R,N)) \]
whose colimit gives the topological Andr\'e-Quillen homology of $R$ with
coefficients in $N$ as defined in \cite{Ba}.

Topological Hochschild homology, $\THH$, of rings of integers in
number fields
is  well-understood: B\"okstedt \cite{boe} calculated $\THH$ of the
integers and in \cite{LM} the general case is covered. The aim of this
paper is the calculation of higher order topological Hochschild
homology of rings of integers in number fields with coefficients in a suitable
residue field.

If $N=R$, then we will abbreviate $\thh{n}(R,R)$ as 
$\thh{n}(R)$. If  $A$ is a
commutative discrete or simplicial ring and $M$ an $A$-module, we
abuse notation and write $\thh{n}(A)$ 
or $\thh{n}(A,M)$ 
for $\thh{n}(HA)$ or $\thh{n}(HA,HM)$, respectively, where $H$ is the
Eilenberg-Mac\,Lane spectrum functor.

John Rognes' redshift conjecture relates algebraic K-theory to 
chromatic phenomena: The stable homotopy category has a 
chromatic filtration and roughly the conjecture states that
algebraic K-theory of a commutative ring spectrum of chromatic level
$n$ is of chromatic level $n+1$, thus algebraic K-theory causes a
chromatic redshift. See \cite[Conjecture 1.3]{ausonirognes} for
a statement of the conjecture.

Replacing spheres by tori gives rise to torus homology and there is work by
Brun, Carlsson, Douglas, Dundas, Rognes, Veen \cite{BCD,cdd,Ve} and others on
this. The hope 
is that, like algebraic K-theory, the homotopy fixed points with
respect to the torus action of torus homology will 
also reveal red-shift phenomena. Positive results in that direction were
obtained by Rognes and
Veen \cite{Ve} and there is ongoing work by Ausoni and Dundas
\cite{ausonidundas}  extending these results to all chromatic
layers. Both approaches use cellular decompositions of tori, and rely
on concrete calculations for each cell, showing that 
higher topological Hochschild homology can be important for
understanding chromatic phenomena.

\subsubsection*{Statement of results}
In the following we give a brief description of our results and
explain how one can use them together with local-to-global techniques and
Bockstein spectral sequences in order to determine 
higher topological Hochschild homology of rings of integers in number
fields. 

We will phrase our results in terms of
iterated $\Tor$ groups:

\begin{Def}\label{defn:Bn}
 Let $k$ be a field. Let $B^1_k(x)$ be the
  polynomial algebra $P_k(x)$ on a generator $x$ in degree $2m$.
  Inductively, we define the $k$-algebra
  $B_k^n(x)=\Tor^{B_k^{n-1}(x)}_*(k,k)$.
\end{Def}
By Cartan \cite{C}, $B_k^2(x)$ is the exterior algebra
$\Lambda_k(\sigma x)$ on a
generator $\sigma x$ in degree $2m+1$; after that, we get a divided
power algebra on a generator of degree $2m+2$, and after that, the
formulas become more complicated (but see \cite{wit} for an
illustration of what the $B_k^n(x)$ look like when $k$ is a finite
field).

\medskip \noindent
In the following, all tensors are over $\bF_p$ unless otherwise
explicitly marked.  Our first result is:

\medskip\noindent{\bf Theorem~\ref{thm:thhnhz}} 
  Let $n\geq1$.  Then $\thh{n}_*(\Z,\bF_p)\cong
  B_{\bF_p}^{n}(x)\otimes B_{\bF_p}^{n+1}(y)$ where $|x|=2p$ and
  $|y|=2p-2$.

\medskip
\noindent
In Remark \ref{rem:thhfp} below we explain how to adapt our argument
to show that 
\[ \thh{n}_*(\bF_p)\cong B_{\bF_p}^{n}(\mu) \]
for a generator $\mu$ of degree two. This result was first obtained by
Basterra and Mandell by other methods. 

\medskip
The calculation of  $\thh{n}_*(\Z,\bF_p)$ along with a calculation of
the Bockstein spectral sequence on it 
would give us $\thh{n}_*(\Z)$. We use this technique in \cite{dlr2} to
determine $\thh{2}_*(\Z_{(p)})$ as a graded commutative ring (see
\cite[Theorem 2.1]{dlr2}). An additive identification of
$\thh{2}_*(\Z)$ is obtained in \cite[Theorem 1.6]{klang}, using
factorization homology and the
identification of $H\Z$ as a Thom spectrum with an algebra structure
over the little $2$-cubes operad.

For a calculation of $\thh{n}(A)$ for more general number rings $A$,
observe first that for 
any $n$ and any commutative ring $A$, $\thh{n}_0(A)\cong A$, since in
the definition of $\thh{n}(A)$, both $d_0$ and $d_1$ have to multiply
all copies of $HA$ indexed on all the $1$-simplices into the copy that
sits over the basepoint, and so by the commutativity of $HA$, $d_0$
and $d_1$ induce the same maps on homology.

The B\"okstedt spectral sequence for higher topological Hochschild
homology (Proposition~7.2 in \cite{wit}) with rational coefficients is
a spectral sequence  with
\[ E^2_{*,*} = \hh{n}(H_*(HA;\bQ))\Rightarrow
 H_*(\thh{n}(A); \bQ). \]
Since $H_*(HA;\bQ)$ consists just of
 $A\otimes \bQ$ in dimension zero and since for a number ring $A$,
 $A\otimes \bQ$ is \'etale over $\bQ$, by Theorem~9.1(a) of
 \cite{wit}, for $*>0$, $\thh{n}_*(A)$ consists entirely of torsion.

 However, since $A$ is a number ring, and hence a Dedekind domain, any
 finitely generated torsion module over it is a finite direct sum of
 modules $A/P_i^{k_i}$, with $P_i$ nonzero prime ideals and $k_i\geq
 1$.
 For each such prime ideal $P_i$, there is a unique prime $p\in\Z$
 for which $pA\subseteq P_i$, and if we consider
 $\thh{n}(A)^\wedge_p$, its homotopy groups in positive dimensions
 will be all the modules
  $A/P_i^{k_i}$ where  $pA\subseteq P_i$, and none of the others.
The methods of Addendum 6.2 in \cite{hm}, which show that for
  any number ring $A$, $\THH(A)^\wedge_p\simeq
  \THH(A^\wedge_p)^\wedge_p$, also show that for general $n\geq 1$,
\[\thh{n}(A)^\wedge_p\simeq \thh{n}(A^\wedge_p)^\wedge_p. \] 
So in order to understand the $P_i$ torsion in $\thh{n}(A)$, we could
see it instead in 
\[\thh{n}(A^\wedge_p)\cong
\thh{n}\left(\prod_{P_i\ \mathrm{prime},\  pA\subseteq P_i} A^\wedge_{P_i}\right)
\simeq
\prod_{P_i\ \mathrm{prime},\  pA\subseteq P_i} \thh{n}(A^\wedge_{P_i}).
\] 
Then,  like calculating $\thh{n}_*(\Z,\bF_p)$ was an intermediate
goal in the calculation of $\thh{n}_*(\Z)$, calculating
$\thh{n}_*(A^\wedge_{P_i},(A^\wedge_{P_i})/P_i)$ is an intermediate
goal in the calculation of $\thh{n}_*(A^\wedge_{P_i})$.

We calculate the groups
$\thh{n}_*(A^\wedge_{P_i},(A^\wedge_{P_i})/P_i)$ below, obtaining

\medskip\noindent{\bf Theorem~\ref{numberring}}
Let $A$ be the ring of integers in a number field,  and let $P$ be a
nonzero prime ideal in $A$. Denote the
residue field
$A/P$ by $\F_q$. Then

\[ \thh{n}_*(A^\wedge_P, A/P)\cong B_{\F_q}^{n}(x_P)\otimes_{\F_q}
B_{\F_q}^{n+1}(y_P) \] 
where
\begin{enumerate}
\item[(i)] $|x_P| =2$ and $|y_P|=0$ if $A$ is  ramified over $\Z$ at $P$, and

\item[(ii)] $|x_P|=2p$ and   $|y_P|=2p-2$, if $A$ is unramified over
  $\Z$ at $P$.
\end{enumerate}
This gives  the homotopy groups of
\[ \thh{n}(A,A/{P_i})\simeq
\thh{n}(A^\wedge_{P_i})\wedge _{H(A^\wedge_{P_i})} H(A/P_i). \] 
As in the $n=1$ case, multiplying $H\Z$ into the copy of
$H(A^\wedge_{P_i})$ over the basepoint shows that
$\thh{n}(A^\wedge_{P_i})$ is a retract of $H\Z\wedge
\thh{n}(A^\wedge_{P_i})$, and so additively, it is a product of
Eilenberg-Mac\,Lane spectra.  Any shifted copy $H(A/P_i^{k_i})$,
$k_i\geq 
1$, that we have in $\thh{n}(A^\wedge_{P_i})$ will yield two
correspondingly shifted copies of $H(A/P_i)$ (one with the same shift,
one with that shift plus one) in $\thh{n}(A^\wedge_{P_i})\wedge
_{H(A^\wedge_{P_i})} H(A/P_i).$  Again  one can then read off the rank
of the $P_i$-torsion from $\thh{n}(A^\wedge_{P_i},A/{P_i})$, and to
understand what the torsion actually is, one would need to look at
Bockstein-type operators associated with multiplication by a
uniformizer of $A^\wedge_{P_i}$.

\section{Identifying square zero extensions}
Let $k$ be a commutative ring, and let $Hk$ be the associated
Eilenberg-Mac~Lane commutative ring spectrum.
We will show that there is exactly one homotopy type of augmented
commutative $Hk$-algebras $C$ with homotopy $\pi_*C\cong\Lambda_k(x)$
where $x$ is a generator in a given positive degree.  That is, there
is a chain of stable equivalences of augmented commutative
$Hk$-algebras $C\simeq Hk\vee\Sigma^mHk$ where the $Hk$-module
$\Sigma^mHk$ is the $m$-fold suspension of $Hk$ and $Hk\vee\Sigma^mHk$
is the square-zero extension of $Hk$ by  $\Sigma^mHk$.

We learned of such a fact  when $k=\F_p$ from Michael Mandell who
proves it by means of topological Andr\'e-Quillen homology and uses it
in a program joint with Maria Basterra.

\begin{prop}\label{square0exterior}
Let $C$ be a commutative augmented $Hk$-algebra and assume that there
is an isomorphism of graded commutative $k$-algebras
$\pi_*C\cong \Lambda_k(x)$ where $|x|=m>0$.  Then there is
a chain of stable equivalences of commutative augmented $Hk$-algebras
between $C$  and $Hk\vee\Sigma^mHk$.
\end{prop}

\begin{proof}
For concreteness, we formulate the proof in symmetric spectra.  Let
$S$ be the sphere spectrum, and $P_S$ and $P_{Hk}$ the free
commutative algebra functors (adjoint to the forgetful functor with
values in $S$ or $Hk$-modules).
We may assume that $C$ is fibrant in the positive $Hk$-model
structure of Shipley \cite{Shipley}.   
Let $M$ be a positively
$Hk$-cofibrant resolution of
$\Sigma^mHk$; for concreteness $M=Hk\smsh F_1(\ess^{m+1})$, where $F_1$ is
the adjoint to the evaluation on the first level.  Represent $x$ by
an $Hk$-module morphism $M\to C$.  Now,
\[ P_{Hk}(M)=\bigvee_{n\geq 0}(M^{\smsh_{Hk} n})_{\Sigma_n}
\cong Hk\smsh P_{S}(F_1(\ess^{m+1})) \]
is the free commutative $Hk$-algebra on $M$.
Since
\[(F_1(\ess^{m+1})^{\smsh n})_{\Sigma_n}\simeq
(F_1(\ess^{m+1})^{\smsh n})_{h\Sigma_n} \]
is $(mn-1)$-connected and $m>0$, the induced
map  $f\colon P_{Hk}(M)\to C$ is  $2m\geq (m+1)$-connected.

Taking the $m$th Postnikov section $P_m$ of $P_{Hk}(M)$ in the setting
of commutative $Hk$-algebra spectra (as done in the \cite{EKMM}
setting in \cite[\S 8]{Ba})
gives a map of commutative $Hk$-algebra spectra
$P_m\to C,$
which is an isomorphism on the non-zero homotopy groups in degree $0$ and
$m$. As both spectra are semistable \cite[4.48]{schwede}, this map is a
stable equivalence of symmetric spectra.
Hence, $C$ and $P_m$ are of the same stable homotopy type, and
repeating the argument with a fibrant model for $Hk\vee\Sigma^mHk$ we
get the promised chain of stable equivalences connecting $C$ and
$Hk\vee\Sigma^mHk$.
%
% The square-zero rings find their root.
% Now, hear they triumphantly hoot:
% "The fog sud'nly clears
% once you know there are spheres!
% (and Postnikov towers to boot)."
\end{proof}

\section{The calculation of $\THH^{[n]}_*(\Z,\F_p)$}
Our goal in this section is to prove

\begin{theo} \label{thm:thhnhz}
  Let $n\geq1$ and $p$ be any prime.  Then \[\thh{n}_*(\Z,\F_p)\cong
    B_{\F_p}^{n}(x)\otimes B_{\F_p}^{n+1}(y)\]
  where $|x|=2p$ and
  $|y|=2p-2$.
\end{theo}
To this end we use the iterative description of $\THH^{[n]}$: the
$n$-sphere $\mathbb{S}^n$ can be decomposed into two hemispheres whose
intersection is the equator, $\mathbb{S}^n = \mathbb{D}^n
\cup_{\mathbb{S}^{n-1}} \mathbb{D}^n$. This decomposition yields
  \cite{Ve} that
\begin{equation} \label{eq:decomp}
\THH^{[n]}(\Z,\F_p) \simeq H\F_p \wedge^L_{\THH^{[n-1]}(\Z,\F_p)}
H\F_p
\end{equation}
where $\wedge^L$ denotes the derived smash product.

Note that for any commutative ring spectrum $R$, $\THH(R)$ is a
commutative $R$-algebra
spectrum \cite[IX.2.2]{EKMM}, so in particular, $\THH(\Z)$ is a
commutative $H\Z$-algebra spectrum.
As $\THH(\Z,\F_p) = \THH(\Z) \wedge_{H\Z} H\F_p$, we have a
commutative $H\F_p$-algebra structure
on $\THH(\Z,\F_p)$ and the multiplication on $H\Z$ and the
$H\Z$-module structure of $H\F_p$ give rise
to a canonical augmentation from $\THH(\Z,\F_p)$ to $H\F_p$.
Thus $\THH(\Z,\F_p)$ is a commutative augmented $H\F_p$-algebra spectrum.

B\"okstedt \cite{boe} calculated $\THH_*(\Z)$ and his result gives the
$n=1$ case of the theorem,
\[ \THH_*(\Z,\F_p) \cong \F_p[x_{2p}] \otimes \Lambda(z_{2p-1}) \]
since $B^2_{\F_p}(y_{2p-2})$ is isomorphic to $\Lambda(z_{2p-1})$.
We use the $(2p-1)$st Postnikov section of commutative augmented
$H\F_p$-algebras  to map  $ \THH_*(\Z,\F_p)$ to something which by
Proposition~\ref{square0exterior}  has to be weakly equivalent to
$H\F_p \vee \Sigma^{2p-1}H\F_p$.
Then we consider the homotopy pushout diagram in the category of
commutative augmented $H\F_p$-algebras
\[ \xymatrix{
{\THH(\Z,\F_p)} \ar[r] \ar[d] & {H\F_p \vee \Sigma^{2p-1}H\F_p}\ar[d]^f \\
{H\F_p} \ar[r]& {(H\F_p \vee \Sigma^{2p-1}H\F_p) \wedge^L_{\THH(\Z,\F_p)}
H\F_p.}
}\] 
A bar spectral sequence argument tells us that the homotopy groups of
$(H\F_p \vee \Sigma^{2p-1}H\F_p) \wedge^L_{\THH(\Z,\F_p)}
H\F_p$  are isomorphic to
$\Lambda (y_{2p+1})$
and using Proposition~\ref{square0exterior}
again we see that the homotopy pushout is a  commutative augmented
$H\F_p$-algebra which is weakly equivalent  to $H\F_p \vee \Sigma^{2p+1} H\F_p$.

The published version of this article unfortunately contains a mistake in the
proof of Lemma 3.2 there.  A variant of the proof which we present below was
known to us at the time, but we sadly submitted a ``simplification''.
The former Lemma 3.2 
states (correctly) that the map 
\[ f \colon  H\F_p \vee \Sigma^{2p-1} H\F_p \to H\F_p \vee \Sigma^{2p+1} H\F_p\]
of augmented commutative $H\F_p$-algebras from above factors in the homotopy
category
through the augmentation $H\F_p \vee \Sigma^{2p-1} H\F_p \to H \F_p $, a fact
which is a special case of the following result:

\begin{lemma} \label{lem:degrees}
  Let $k$ be a discrete commutative ring, and let $a\neq b$ be two positive
  integers such that $b\leq 2a$.  Then any map in the category of commutative
  augmented $Hk$-algebras between the %trivial
  square-zero extensions
\[Hk\vee \Sigma^a Hk \to Hk\vee \Sigma^b Hk \] 
factors in the homotopy category of commutative augmented $Hk$-algebras
through the augmentation $Hk\vee \Sigma^a Hk
\to Hk$.
\end{lemma}
We use the notation of Basterra \cite{Ba}.  For  a commutative
$\mathbb{S}$-algebra $R$, let $\calM_R$ denote the category of $R$-modules,
$\calN_R$ denote the category of non-unital commutative $R$-algebras, and
$\calC_{R/R}$ denote the category of commutative augmented $R$-algebras.
Consider the adjunctions
\[ \xymatrix{\calN_R\ar@<0.5ex>[r]^Q&\ar@<0.5ex>[l]^Z\calM_R\ar@<0.5ex>[r]^A&\ar@<0.5ex>[l]^U\calN_R\ar@<0.5ex>[r]^K&\ar@<0.5ex>[l]^I\calC_{R/R},}
\] 
where $U$ is the forgetful functor with left adjoint $A$, sending an
$R$-module $M$ to the free non-unital commutative algebra on $M$, 
$UAM=\bigvee_{j>0} M^{\wedge_R j}/\Sigma_j$
with multiplication induced by concatenation; $Z$ is the square zero extension
making an $R$-module $M$ into a nonunital commutative $R$-algebra $ZM$ using
the trivial multiplication map sending everything to the basepoint, with left
adjoint the indecomposables $Q$, sending a non-unital $R$-algebra to the
pushout $QC$ of
$\xymatrix{{*}&\ar[l]U(C\wedge_RC)\ar[r]^-{\text{mult.}}&UC;}$ 
and, finally,
$K$ is given by $K(M)=R\vee M$, with multiplication using $R$'s
multiplication, $M$'s multiplication, and $M$'s $R$-module structure, with
right adjoint the augmentation ideal $I$, sending an augmented commutative
$R$-algebra $B$ to the pullback $I(B)$ of
$\xymatrix{{*}\ar[r]&R&\ar[l]_{\text{augm.}}B.}$

If $C$ is a $q$-cofibrant commutative $R$-algebra and $N\in\calM_R$, then we get
that the canonical equivalence $ZN\we IKZN$ induces an equivalence
\[ \calC_{R/R}(KC,KZN)\cong \calN_R(C,IKZN)\ew\calN_R(C,ZN)\cong\calM_R(QC,N).
\] 

\begin{proof}[Proof of Lemma~\ref{lem:degrees}]
  In the special case of the lemma, letting $R$ be $Hk$ and $C\we Z\Sigma^aR$
  a $q$-cofibrant replacement in $\calN_R$, 
we must show that the mapping space $\calM_R(QC,\Sigma^bR)$
is connected. 
Noting that $a\neq b$ implies that any map from $\Sigma^aR\simeq UC$ to
$\Sigma^bR$ is null\-homotopic, we are done once we show the following more
general Proposition~\ref{prop:moregeneral}.
\end{proof}
In what follows $R$ is a connective  $q$-cofibrant commutative $S$-algebra.
\begin{prop}
  \label{prop:moregeneral}
  Let $C\in\calN_R$ have $\pi_\ell UC=0$ for $\ell<m$, let $N\in \calM_R$
  have $\pi_\ell(N)=0$ for $\ell>2m$ and let $f\colon\ C\to ZN$ in the homotopy
  category $ho\calN_R$ of non-unital commutative $R$-algebras be such that the
  underlying $R$-module map $Uf\colon\ UC\to UZN= N$ is nullhomotopic.
  Then $f$ is nullhomotopic.

  In other words (using the isomorphism $\calN_R(C,ZN)\cong\calM_R(QC,N)$),
  the map $UC\to QC$ induces an injection $ho\calM_R(QC,N)\to ho\calM_R(UC,N)$.
\end{prop}
Letting $LC$ be the homotopy cofiber of the map $UC\to QC$, we see that it is
enough to prove that $\calM_R(LC,N)$ is connected, which is immediate from the
two following lemmas (with $X=LC$ and $n=2m$).
\begin{lemma}
  Let $X$ be a $q$-cofibrant $R$-module with $\pi_\ell X=0$ for $\ell\leq n$
  and $N$ an $R$-module with $\pi_\ell N=0$ for $\ell>n$.  Then $\calM_R(X,N)$
  is connected.
\end{lemma}
\begin{proof}
  Since $R$ is connective, we can build $X$ by $R$-cells of dimension $\ell> n$
  and $\calM_R(X,N)$ is the limit of a tower of fibrations with connected
  fibers of the form $\calM_R(\Sigma^\ell R,N)\simeq\Sigma^{-\ell} N$.
\end{proof}

\begin{lemma}
  Let $C$ be a $q$-cofibrant non-unital commutative $R$-algebra such that
  $\pi_\ell UC=0$ for $\ell<m$.  Then the canonical map $UC\to QC$ is
  $2m$-connected.
 \end{lemma}
 \begin{proof}
Let $X\we UC$ be a $q$-cofibrant replacement in $\calM_R$.  Then, by induction 
on \cite[5.3]{Ba}, the induced map $(UA)^tX\to(UA)^tUC\cong Q(AU)^{t+1}C$ is
an equivalence and by \cite[III.5.1]{EKMM} the counit $X\to (AU)^tX$ is
$2m$-connected, so that the canonical map $UC\to |[t]\mapsto Q(AU)^{t+1}C|
\cong Q|[t]\mapsto (AU)^{t+1}C|$ is $2m$-connected (where the simplicial
structure is induced by the adjoint pair).  Since, by  \cite[5.4]{Ba}, the map
$Q|[t]\mapsto (AU)^{t+1}C|\to QC$ is an equivalence, we are done.
 \end{proof}

% \begin{lemma} \label{lem:factor}
% Let 
% \begin{align*}
% 	f\colon H\F_p \vee \Sigma^{2p-1}H\F_p &\to {(H\F_p \vee
%   \Sigma^{2p-1}H\F_p) \wedge^L_{\THH(\Z,\F_p)}
% 	H\F_p.}\\ &\cong H\F_p \vee \Sigma^{2p+1}H\F_p
% \end{align*}
% be the map
% in the homotopy category of augmented commutative $H\F_p$-algebras
% induced by the pushout above.  Then $f$ factors through the
% augmentation $\varepsilon\colon H\F_p \vee \Sigma^{2p-1}H\F_p\to
% H\F_p$.
% \end{lemma}

% \begin{proof}
 
%  Consider the diagram
% $$\xymatrix{
% {H\F_p \vee \Sigma^{2p-1} H\F_p}\ar[r]^{=} \ar[d]^\epsilon & {H\F_p
%   \vee \Sigma^{2p-1}H\F_p} \ar[d]^f \\
% {H\F_p} \ar [r] & {H\F_p \vee \Sigma^{2p+1}H\F_p} ,}$$
% which we obtain from the previous diagram 
% because the augmentation 
% $\THH(\Z,\F_p)\to H\F_p$ has to factor through the Postnikov map to
% $H\F_p \vee \Sigma^{2p-1} H\F_p$ and of course, so does the Postnikov map 
% itself.
% This diagram commutes, so the maps to the bottom right factor through
% the pushout $H\F_p\wedge _{H\F_p \vee \Sigma^{2p-1} H\F_p} ({H\F_p \vee
%   \Sigma^{2p-1} H\F_p}) \simeq H\F_p$ of the top left corner,
% so $f$ factors through
% $\epsilon$.
%\end{proof}
We can therefore describe  $\THH^{[2]}(\Z,\F_p)$ by the following
iterated homotopy pushout diagram.
\[ \xymatrix{
{\THH(\Z,\F_p)} \ar[r] \ar[dd] & {H\F_p \vee \Sigma^{2p-1}H\F_p} \ar[dd]^f
\ar[rr]
\ar[dr]^{\varepsilon} & & {H\F_p} \ar'[d][dd] \ar[dr] & \\
 & &{H\F_p} \ar[dl]_{\eta} \ar[rr] & & {\Gamma} \ar[dl]\\
{H\F_p} \ar [r] & {H\F_p \vee \Sigma^{2p+1}H\F_p} \ar[rr] & &
{\THH^{[2]}(\Z,\F_p)} &
}\]
Here, $\Gamma$ denotes the homotopy pushout of the upper right
subdiagram in the category of commutative $H\F_p$-algebras, and as above
we get
\[ \Gamma = H\F_p \wedge^L_{H\F_p \vee \Sigma^{2p-1}H\F_p} H\F_p. \]
We have again a $\Tor$-spectral sequence converging to the homotopy
groups of the spectrum $\Gamma$ with 
\[E^2_{*,*} = \Tor_{*,*}^{\Lambda(z_{2p-1})}(\F_p,\F_p) \]
and hence $\pi_*(\Gamma)$ is isomorphic to a divided power algebra over
$\F_p$ on a generator in degree $2p$,
$\Gamma(a_{2p})$. In the iterated homotopy pushout diagram all maps
involved are maps of commutative $S$-algebras and thus we can identify
$\THH^{[2]}(\Z,\F_p)$ as a
commutative $H\F_p$-algebra
as
\begin{align*}
  \label{eq:thh2z}
\THH^{[2]}(\Z,\F_p) &\simeq (H\F_p \vee \Sigma^{2p+1}H\F_p)
                      \wedge^L_{H\F_p \vee \Sigma^{2p-1}H\F_p} H\F_p \\
  &\simeq
    (H\F_p \vee \Sigma^{2p+1}H\F_p)\wedge_{H\F_p}\Gamma\\
  &\cong \Gamma \vee
\Sigma^{2p+1}\Gamma.
\end{align*}
Thus, $\THH^{[2]}(\Z,\F_p)$ is equivalent to the
bar construction
\[B_{H\F_p}(H\F_p, H\F_p \vee \Sigma^{2p-1}H\F_p, H\F_p  \vee \Sigma^{2p+1}H\F_p). \]
Its homotopy groups are
\[ \THH_*^{[2]}(\Z,\F_p) \cong \Gamma(a_{2p}) \otimes \Lambda(y_{2p+1}). \]
We use this to determine higher $\THH$  via  iterated bar constructions.
 We know that
$\THH^{[n+1]}(H\Z,H\F_p)$ is equivalent to the derived smash product
\[ H\F_p \wedge^L_{\THH^{[n]}(H\Z,H\F_p)} H\F_p \]
whose homotopy groups are the ones of the bar construction
\[ B_{H\F_p}( H\F_p, \THH^{[n]}(H\Z,H\F_p), H\F_p)\] and iteratively, we can express
$\THH^{[n]}(H\Z,H\F_p)$ again in terms of such a bar construction as
long as $n$ is
greater than two. For $n=2$ we know the answer by the above
argument. For larger $n$  we can
determine the homotopy groups of $\THH^{[n+1]}(H\Z,H\F_p)$
iteratively.
Abbreviating $H\F_p \vee \Sigma^{2p-1}H\F_p$ to $E(z)$ and
$H\F_p \vee \Sigma^{2p+1}H\F_p$ to $E(y)$ we define
\[ B^{(n)} := B_{H\F_p}(H\F_p, \ldots, B_{H\F_p}(H\F_p, B_{H\F_p}(H\F_p, E(z), E(y)),H\F_p),
\ldots, H\F_p) \]
with $n-1$ pairs of outer terms of $H\F_p$.
We denote by $\underline{E(y)}$ the constant simplicial
$H\F_p$-algebra spectrum on $E(y)$.
\begin{lemma}
As $n$-simplicial commutative $H\F_p$-algebras
\[ B^{(n)} \simeq B^{(n)}_{H\F_p}(H\F_p, E(z), H\F_p) \wedge_{H\F_p}
B^{(n-1)}_{H\F_p}(H\F_p, \underline{E(y)}, H\F_p) \] 
for all $n \geq 2$.
\end{lemma}
\begin{proof}
We show the claim directly for $n=2$: $B^{(2)}$ is
\[ B_{H\F_p}(H\F_p, B_{H\F_p}(H\F_p, E(z), E(y)), H\F_p). \] 
As we know from Lemma~\ref{lem:degrees} that the $E(z)$-module
structure of $E(y)$ factors via the augmentation map through the
$H\F_p$-module structure of $E(y)$, we get that
$B_{H\F_p}(H\F_p, E(z), E(y))$ can be split  as an augmented
simplicial commutative
$H\F_p$-algebra as $B_{H\F_p}(H\F_p, E(z),
H\F_p) \wedge_{H\F_p} \underline{E(y)}$ and thus we get a weak
equivalence of bisimplicial commutative $H\F_p$-algebra spectra:
\begin{align*}
 &\quad\ B_{H\F_p}(H\F_p, B_{H\F_p}(H\F_p, E(z), E(y)), H\F_p) \\
	&\simeq  B_{H\F_p}(H\F_p, B_{H\F_p}(H\F_p, E(z),
H\F_p) \wedge_{H\F_p} \underline{E(y)}, H\F_p) \\
	&\simeq  B_{H\F_p}(H\F_p, B_{H\F_p}(H\F_p, E(z),
H\F_p), H\F_p)  \wedge_{H\F_p} B_{H\F_p}(H\F_p, \underline{E(y)}, H\F_p)\\
	&=  B^{(2)}_{H\F_p}(H\F_p, E(z), H\F_p) \wedge_{H\F_p} B_{H\F_p}(H\F_p,
\underline{E(y)}, H\F_p).
\end{align*}
For the second weak equivalence in the chain above, we have used that the bar
construction 
$B_R(R, X \!\wedge_R Y, R)$ of the smash product of two commutative
simplicial $R$-algebra spectra $X$ and $Y$ is equivalent as a
bisimplicial commutative $R$-algebra to $B_R(R,X,R) \wedge_R
B(R,Y,R)$.

By induction we assume that $n$ is bigger than $2$ and that we know
the result for all $k < n$. Then
\begin{align*}
 B^{(n)} & = B_{H\F_p} (H\F_p, B^{(n-1)}, H\F_p) \\
 & \simeq B_{H\F_p}(H\F_p, B^{(n-1)}_{H\F_p}(H\F_p, E(z), H\F_p)\\
&\quad \wedge_{H\F_p}
 B^{(n-2)}_{H\F_p}(H\F_p, \underline{E(y)}, H\F_p), H\F_p) \\
 &  \simeq B_{H\F_p}(H\F_p, B^{(n-1)}_{H\F_p}(H\F_p, E(z), H\F_p) , H\F_p)
\\
&\quad \wedge_{H\F_p} 
 B_{H\F_p}(H\F_p, B^{(n-2)}_{H\F_p}(H\F_p, \underline{E(y)},
 H\F_p), H\F_p) \\
 & = B^{(n)}_{H\F_p}(H\F_p, E(z), H\F_p)
 \wedge_{H\F_p} B^{(n-1)}_{H\F_p}(H\F_p, \underline{E(y)}, H\F_p).
\end{align*}
\end{proof}
We view $\THH^{(n)}(H\Z, H\F_p)$ as a simplicial commutative $H\F_p$-algebra
for all $n\geq 1$ and therefore describe $\THH^{(n+1)}(H\Z, H\F_p)$
as the diagonal of the bar construction
$B_{H\F_p}( H\F_p, \THH^{[n]}(H\Z,H\F_p), H\F_p)$.
\begin{cor}
We obtain, that
\begin{align*}
	&\quad\ \THH^{(n+1)}_*  (H\Z, H\F_p)  \\ 
&\cong \pi_*\text{diag} B^{(n)}_{H\F_p}(H\F_p, E(z),
H\F_p) \otimes_{\F_p} \pi_* \text{diag} B^{(n-1)}_{H\F_p}(H\F_p,
\underline{E(y)}, H\F_p).
\end{align*}
\end{cor}
For sake of definiteness in the following we will work in the category
of symmetric spectra in simplicial sets, $\Sp^\Sigma$ \cite{hss}.
The Eilenberg-Mac\,Lane spectrum gives rise to a functor
\[ H\colon \sab \ra \Sp^\Sigma \]
such that $HA(n) = \mathrm{diag}(A\otimes \tilde{\Z}(S^n))$ where
$S^n = (S^1)^{\wedge n}$ and  $\tilde{\Z}(-)$ denotes the free
abelian group generated by all non-basepoint elements. This functor is lax
symmetric monoidal \cite[2.7,3.11]{schwede}. A square-zero extension
$H\F_p \vee \Sigma^n H\F_p$ (for $n\geq 1$) can be modelled
by $\F_p(\Delta_n/\partial \Delta_n)$:

\begin{lemma}
There is a stable equivalence of commutative symmetric ring spectra $\psi\colon
H\F_p(\Delta_n/\partial \Delta_n) \ra H\F_p \vee \Sigma^n H\F_p$.
\end{lemma}
\begin{proof}
There are two non-degenerate simplices in $\Delta_n/\partial
\Delta_n$: a zero-simplex $*$ corresponding to the unique basepoint and an
$n$-simplex corresponding to the identity map $\id_{[n]}$ on the set
$[n]=\{0,\ldots,n\}$. We can represent any simplex in  $\Delta_n/\partial
\Delta_n$ as $s_{i_\ell} \circ \cdots \circ s_{i_0}(*)$ or $s_{i_\ell}
\circ \cdots \circ s_{i_0}(\id_{[n]})$. We define the map $\psi(m)_\ell$ from
$H\F_p(\Delta_n/\partial \Delta_n)(m)_\ell = \F_p(\Delta_n/\partial
\Delta_n)_\ell \otimes \tilde{\Z}(S^m_\ell)$ to $(H\F_p \vee \Sigma^n
H\F_p)(m)_\ell= \F_p\otimes \tilde{\Z}(S^m_\ell) \vee \Delta_n/\partial
\Delta_n \wedge \F_p\otimes \tilde{\Z}(S^m_\ell)$
on generators by setting
\begin{align*}
\psi(s_{i_\ell} \circ \cdots \circ s_{i_0}(*) \otimes x) & = 1 \otimes
x \in \F_p\otimes \tilde{\Z}(S^m_\ell) \\
\psi(s_{i_\ell} \circ \cdots \circ s_{i_0}(\id_{[n]}) \otimes x) & =
[s_{i_\ell} \circ \cdots \circ s_{i_0}(\id_{[n]}), x] \in \Delta_n/\partial
\Delta_n \wedge \F_p\otimes \tilde{\Z}(S^m_\ell)
\end{align*}
for $x \in S^m_\ell$ and by extending it in a bilinear manner. This
map is well-defined and multiplicative. Both spectra have finite
stable homotopy groups and are therefore semistable. It thus suffices
to show that $\psi$ is a stable homotopy equivalence. The
stable homotopy groups on both sides are exterior algebras on a generator
in degree $n$ and $\psi$ induces the map on stable homotopy groups that sends
$1$ to $1$ and maps the generator in degree $n$ to a degree $n$ generator.
\end{proof}

\medskip
We have a weak equivalence
\[ B_{H\F_p}(H\F_p, HA, HC) \ra  H(B_{\F_p}(\F_p, A, C)) \] 
for all simplicial $\F_p$-algebras $A$ and $C$.
Let $N$ denote the normalization functor from simplicial $\F_p$-vector
spaces
to non-negatively graded chain complexes over $\F_p$. This is a lax symmetric
monoidal functor, so it sends simplicial commutative $\F_p$-algebras to
commutative differential graded $\F_p$-algebras. Note that we obtain
isomorphisms of differential graded $\F_p$-algebras
\[  N(\F_p(\Delta_{2p-1}/\partial \Delta_{2p-1})) \cong \Lambda_{\F_p}(z),
  \qquad
N(\F_p(\Delta_{2p+1}/\partial \Delta_{2p+1})) \cong \Lambda_{\F_p}(y) \]
because for positive $n$, $\Delta_n/\partial \Delta_n$ has only a
non-degenerate zero cell and a non-degenerate $n$-cell. Note that
$B^{(n)}_{\F_p}(\F_p, \F_p(\Delta_{2p-1}/\partial \Delta_{2p-1}), \F_p)$ is
an $(n+1)$-fold
simplicial commutative $\F_p$-algebra. We can calculate the homotopy
groups of its diagonal as the homology of the total complex associated
to the bicomplex
\begin{align*}
	C_{r,s} &= N_r\text{diag}^n B^{(n)}_{\F_p}(\F_p,
N_s\F_p(\Delta_{2p-1}/\partial \Delta_{2p-1}),
              \F_p)\\
	& \cong N_r\text{diag}^n B^{(n)}_{\F_p}(\F_p,
  \F_p(\Delta_{2p-1}/\partial \Delta_{2p-1})_s, \F_p).
\end{align*}
As the differential in $s$-direction is trivial the spectral sequence
collapses at the $E^2$-term with total homology isomorphic to the
given by the homology
of the differential graded $n$-fold bar construction. These homology
groups were calculated in \cite{wit} and we obtain that
$\pi_*\text{diag}B^{(n)}_{H\F_p}(H\F_p, E(z),
H\F_p) \cong B^{n+2}_{\F_p}(y)$. Similarly,
$\pi_*\text{diag}B^{(n-1)}_{H\F_p}(H\F_p, \underline{E(y)}, H\F_p)
\cong B^{n+1}_{\F_p}(x)$. This proves Theorem~\ref{thm:thhnhz}.

\begin{remark} \label{rem:thhfp}
Note that a similar argument as above shows that
\begin{equation} \label{eq:thhfp}
\THH^{[n]}_*(\F_p) \cong B^n_{\F_p}(\mu)
\end{equation}
for all $n \geq 1$ and all primes $p$. Here the degree of $\mu$ is
two. For $n=1$ this 
calculation is due to B\"okstedt
\cite{boe}.  The only difference from the argument
above is that the first step is slightly easier: in order to calculate
$\THH^{[2]}(\F_p)$ we do not have to split off the Postnikov section --
the bar spectral sequence immediately yields that $\THH^{[2]}_*(\F_p)$
is an exterior algebra on a class in degree $3$.
By Proposition \ref{square0exterior} we know that we
can model $\THH^{[2]}(\F_p)$ as  $H\F_p \vee \Sigma^3H\F_p$. From there on
the argument is completely analogous to the one above. Maria Basterra and
Mike Mandell proved the isomorphism \eqref{eq:thhfp} for all $n$ and
$p$ in 1999 using an 
$E_\infty$ bar construction argument. Partial results about higher
$\THH$ of $\F_p$ were 
obtained by Rognes, Veen \cite{Ve} and in \cite{wit}.

\end{remark}

\section{The number ring case}

The aim of this section is to prove Theorem~\ref{numberring},
calculating the higher topological Hochschild homology of number rings
with coefficients in the residue field.
The calculation starts with the following observation.

\begin{lemma}\label{padicHH}
Let $B$ be a characteristic zero complete discrete valuation ring with
residue field $\F_q$ of characteristic $p>0$ and let $P$
denote the ideal consisting of all elements with positive valuation in $B$.
Then 
\[ \HH_1^{\Z_p}(B, B/P)\cong\begin{cases}
B/P, &\text{ if } B \text{ is ramified over } \Z_p \text{ at } P,\\
0, & \text{otherwise.}\end{cases} \] 
\end{lemma}

\begin{proof}
By Proposition 12~in Chapter~3 of \cite{serre}, $B$ is generated over
$\Z_p$ by a single element, $B=\Z_p[x]/(f(x))$ for some monic
polynomial $f$.  By a well-known calculation which can be traced back
to \cite{tate}, for any ring $R$ and monic polynomial $f(x)$ over it,
$\HH^R_1(R[x]/(f(x)))\cong R[x]/(f(x), f'(x))$.   By Corollary~2 of
Chapter~3 of \cite{serre}, the ideal $(f'(x))$ in $\Z_p[x]/(f(x))$ is
equal to the different ideal $\mathcal{D}_{B/\Z_p}$, so the result for
$B$ over $\Z_p$ becomes
\[ \HH^{\Z_p}_1(B)\cong B/\mathcal{D}_{B/\Z_p}. \] 
Since $B$ is commutative, we also know that $\HH^{\Z_p}_0(B)\cong B$ is
free over $B$, hence
 $\Tor_{s}^B(\HH^{\Z_p}_0(B), B/P)=0$ for $s>0$. If we tensor the Hochschild
complex of $B$ over $B$ with $B/P$, then we get by the universal
coefficient theorem, that
\[ \HH^{\Z_p}_1(B, B/P)\cong \HH^{\Z_p}_1(B)\otimes_B B/P\cong
B/\mathcal{D}_{B/\Z_p} \otimes_B B/P\cong B/(\mathcal{D}_{B/\Z_p},
P). \] 
If $\mathcal{D}_{B/\Z_p} \subseteq P$, this is just $B/P$, but if not,
by the maximality of $P$ in $B$ we get that  $B/(\mathcal{D}_{B/\Z_p}, P)\cong 0$.
Theorem~1 in Chapter~3 of \cite{serre} says that an extension $B$
of $\Z_p$ is ramified at an ideal $P$ of $B$ if and only if $P$
divides the different ideal $\mathcal{D}_{B/\Z_p}$.
\end{proof}

From this we establish the one-dimensional (ordinary topological
Hoch\-schild homology) case of Theorem~\ref{numberring},
which is
closely related to Theorem~4.4 in \cite{LM}, and in fact in the unramified
case is exactly Theorem~4.4 (i) there (since in the unramified case $P=pA$
for a rational prime $p$).  The symbol $x_i$ in the statement
indicates a generator of degree $i$.

\begin{prop}
\label{numberringTHH}
Let $A$ be the ring of integers in a number field, and let $P$ be a
nonzero prime ideal in $A$.
Denote the residue field $A/P$ by $\F_q$.
\begin{enumerate}[(i)]
\item  If  $A$ is  ramified over $\Z$ at $P$,
  $\thh{1}_*(A^\wedge_P, \F_q)\cong \F_q[x_2]
  \otimes_{\F_q}\Lambda_{\F_q}[x_1]$.
\item  If $A$ is unramified over $\Z$ at $P$, $\thh{1}_*(A^\wedge_P,\F_q)\!\cong\!
  \F_q[x_{2p}] \!\otimes_{\F_q}\! \Lambda_{\F_q}[x_{2p-1}]$.
  \end{enumerate}
\end{prop}

\begin{proof}
We set $B=A^\wedge_P$, to get a ring that satisfies the
conditions of\linebreak Lemma~\ref{padicHH}.
We now use
$P$ to denote the ideal in $B$ obtained as $PB$ for the ideal $P$ of
$A$.  We use Morten Brun's spectral sequence \cite[p.~30]{brun} from
Theorem 3.3 in \cite{LM} for the map $B\to B/P$.  This gives a multiplicative
spectral sequence
\[ E^2_{r,s}=\THH_r(B/P, \Tor_s^B(B/P, B/P))\Rightarrow \THH_{r+s} (B, B/P).
\] 
Since $P$ is a principal ideal in $B$, generated by any uniformizer
$\pi$, the resolution
\[ \xymatrix@1{0 \ar[r] & B \ar[r]^{\cdot \pi} & B\ar[r] & B/P \ar[r] & 0} \]
shows that $\Tor_*^B (B/P, B/P)\cong\Lambda_{\F_q}(\tau_1)$ for a
$1$-dimensional generator $\tau_1$, where $B/P =(A^\wedge_P)/P \cong
A/P=\F_q$.

\smallskip
B\"okstedt showed in \cite{boe} that $\THH_*(\F_p)\cong \F_p[u_2]$,
and since $\HH_*(\F_q)$ consists only of $\F_q$ in dimension zero, the
spectral sequence of Theorem 2.2 in \cite{LSpSq}
\[ E^2_{r,s}=\HH_r^{\F_p} (\F_q, \THH_s(\F_p;\F_q) )\Rightarrow
\THH_{r+s}(\F_q)\] 
consists only of $\F_q\otimes \F_p[u_2]$ in the zeroth row, we get that
\[ \THH_*(\F_q)\cong \F_q[u_2]. \] 
Thus Brun's spectral sequence takes the form
\[ E^2_{*,0}\cong \F_q[u_2],\ \ \ E^2_{*, 1} \cong \tau_1\cdot \F_q[u_2]. \]
From Lemma~\ref{padicHH} and the fact that Hochschild and topological
Hochschild homology agree in degree $1$, we know  that we
end up with nothing in total degree $1$ if $B$ is unramified over
$\Z_p$, and with a copy of $\F_q$ if $B$ is ramified.  So we get
\[ d^2 (u_2)
=\begin{cases}
0, &\text{ if } B \text{ is  ramified over } \Z_p \text{ at } P,\\
\text{(unit)} \cdot\tau_1, & \text{ otherwise.}\end{cases}\] 
In the ramified case we already know that $d^2$ vanishes on $1$
and $\tau_1$, since there is nothing these elements could hit. Therefore,
we get that
$d^2=0$. As $d^2$ is the last differential that could be
nontrivial, $E^\infty_{*,*}\cong E^2_{*,*}\cong
\Lambda_{\F_q}(\tau_1)\otimes_{\F_q}\F_q[u_2]$, and since this is the
multiplication with the fewest relations possible that could be defined on a
graded-commutative algebra with this linear structure, extensions
cannot give any other multiplicative structure and we get
\[ \THH_*(B, B/P)\cong  \Lambda_{\F_q}(\tau_1)\otimes_{\F_q}\F_q[u_2]. \] 

In the unramified case, the knowledge what $d^2$ does on the generators
shows us that $d^2(u_2^a)=\mbox{(unit)}\cdot\tau_1\cdot u_2^{a-1}$
when $p$ does not divide $a$, but $d^2(u_2^{pk})=0$ and nothing hits
the elements $\tau_1\cdot u_2^{pk-1}$.  Again $d^2$ is the last
differential that could be nonzero so $E^\infty_{*,*}\cong
E^3_{*,*}\cong \Lambda_{\F_q}(\tau_1\cdot
u_2^{p-1})\otimes_{\F_q}\F_q[u_2^p]$, and since this is again a
multiplication with the fewest relations possible that could be defined on a
graded-commutative algebra with this linear structure, extensions
cannot give any other multiplicative structure and we get
\[ \THH_*(B, B/P)\cong  \Lambda_{\F_q}(\tau_1\cdot
u_2^{p-1})\otimes_{\F_q}\F_q[u_2^p].\vspace{-1em}\]  \end{proof}

\begin{theo}
\label{numberring}
Let $A$ be the ring of integers in a number field,  and let $P$ be a
nonzero prime ideal in $A$.
Denote the residue field $A/P$ by $\F_q$.

Then
\[ \thh{n}_*(A^\wedge_P, \F_q)\cong B_{\F_q}^{n}(x_P)\otimes_{\F_q}
B_{\F_q}^{n+1}(y_P) \] 
where
\begin{enumerate}
\item[(i)] $|x_P| =2$ and $|y_P|=0$ if $A$ is  ramified over $\Z$ at $P$, and
\item[(ii)] $|x_P|=2p$ and   $|y_P|=2p-2$, if $A$ is unramified over
  $\Z$ at $P$.
\end{enumerate}
\end{theo}

\begin{proof}
The $n=1$ case is true  by Proposition~\ref{numberringTHH},  with
$x=x_2$ and $y$ zero-dimensional (so that
$B^1_{\F_q}(y)\cong\Lambda_{\F_q}(x_1)$)  in the ramified case, and with
$x=x_{2p}$ and $y$ of dimension $2p-2$ (so that
$B^1_{\F_q}(y)\cong\Lambda_{\F_q}(x_{2p-1})$) in the unramified case.

The rest of the proof proceeds by exact analogy to the calculation of
$\THH^{[n]}_*(\Z,\F_p)$.
\end{proof}

\begin{note}
  The ramified case  Theorem~\ref{numberring} (i) can actually be
  proven quite algebraically by noting that for an arbitrary flat ring
  $A$ and $A$-bimodule $M$, the linearization map $\THH(A,M)\to
  H(\HH^\Z(A,M))$ is $3$-connected so that the first Postnikov sections
  of $\THH(A,M)$ and $H(\HH^\Z(A,M))$ agree.  As a matter of fact, when
  $A$ is a $\Z_{(p)}$-algebra, B\"okstedt's calculation of the
  topological Hochschild homology of the integers gives that
  Theorem~2.3 of \cite{LSpSq} implies that this can be improved to
  saying that $\THH(A,M)\to H(\HH^\Z(A,M))$ is $(2p-1)$-connected.
  This means that the Postnikov section involved in the crucial step
  moving from $\THH$ to to the algebraic $\THH^{[2]}$ coincides with
  that of Hochschild homology.  This was how we originally established
  the calculation in the ramified case.
\end{note}

\begin{note}
The unramified case, Theorem~\ref{numberring}(ii), could have also been deduced
from Theorem~\ref{thm:thhnhz} by showing that
\[ \thh{n}_*(A^\wedge_P, A/P)\cong\thh{n}_*(\Z_p, \F_p)\otimes \F_q\quad
	(\text{where } \F_q=A/P) \] 
as an augmented
$\F_q$-algebra, where the augmentation on the right-hand
side comes from the augmentation of $\thh{n}_*(\Z,\F_p)$ over $\F_p$,
tensored with the identity of $\F_q$.
This is true for $n=1$ by
Theorem 4.4 (i) of  \cite{LM}, and then we proceed inductively, using
the decomposition from \eqref{eq:decomp}. This yields in this case
a decomposition
\[ 
\thh{n+1}(A^\wedge_P
,\F_q)\simeq H\F_q\wedge^L_{\thh{n}(A^\wedge_P
,\F_q)} H\F_q.
\] 
and  a multiplicative spectral sequence
\[ \Tor_{*,*}^{\thh{n}_*(A^\wedge_P,\F_q)} (\F_q,\F_q)\Rightarrow
\thh{n+1}_*(A^\wedge_P,\F_q), \] 
which can be rewritten by the inductive hypothesis as
\begin{align*}
	&\quad\ \Tor_{*,*}^{\thh{n}_*(\Z_p, \F_p)\otimes \F_q} (\F_p\otimes
\F_q,\F_p\otimes \F_q)\\
&\cong
\Tor_{*,*}^{\thh{n}_*(\Z_p, \F_p)}(\F_p,\F_p) \otimes \Tor_*^{\F_q}(\F_q,\F_q)\\
	&\cong
\Tor_{*,*}^{\thh{n}_*(\Z, \F_p)}(\F_p,\F_p) \otimes \F_q,
\end{align*}
where the first factor is the image of the $E^2$-term of the spectral
sequence calculating
$\thh{n+1}_*(\Z, \F_p)$ and the second term is in $E^2_{0,0}$ and
therefore can cause no nontrivial differentials or multiplicative
extensions.  This splitting is a splitting of algebras and the
augmentation is that of the first factor  tensored with the identity
of the second.
\end{note}

\section*{Acknowledgements} 
This material is based upon work supported by  
the National Science Foundation  under Grant No. 0932078000 while the
authors were in residence at the Mathematical Sciences Research
Institute in Berkeley California, during the Spring 2014 program on
Algebraic Topology. 

We are grateful to  Mike Mandell for providing the crucial hint
that there is just one homotopy type of augmented commutative
$H\F_p$-algebras with homotopy groups $\Lambda_{\F_p}(x)$.

%\received{April 9, 2015}}
\end{document}